\documentclass[12pt,leqno]{amsart}
\usepackage{amssymb,amsmath,amsthm,bm,scalerel,stackengine}
\usepackage{amssymb,amsmath,amsthm,bm,scalerel,stackengine}
\usepackage{graphicx}
\usepackage{color}
\usepackage[textsize=tiny]{todonotes}
\usepackage[normalem]{ulem}
\usepackage[bookmarksopen,bookmarksdepth=3,colorlinks,citecolor=red,pagebackref,hypertexnames=true]{hyperref}
\usepackage[msc-links,nobysame,non-sorted-cites, initials]{amsrefs}
\usepackage[inline]{enumitem}
\usepackage{mathtools}
\mathtoolsset{showonlyrefs}
\usetikzlibrary{quotes,arrows.meta}

\definecolor{darkblue}{RGB}{0,0,160}

\usepackage[lining]{libertine}   
\usepackage{cabin}
\usepackage[libertine]{newtxmath}

\usepackage[T1]{fontenc}

\def\eps{\varepsilon}
\def\d{{\rm d}}
\def\R {\mathbb{R}}
\def\N {\mathbb{N}}

\def\supp {{\mathrm{supp}\,}}
\def\Z {{\mathbb Z}}
\def\C {{\mathcal C}}

\newcommand{\TT}{\mathbb{ T} }

\newcommand{\cic}{\bm}
\newcounter{counter}

\numberwithin{equation}{section}
\numberwithin{counter2}{section}
\newtheorem{proposition}[subsection]{Proposition}
\newtheorem{theorem}[counter]{Theorem}

\newtheorem{corollary}[counter]{Corollary}
\newtheorem{lemma}[subsection]{Lemma}

\theoremstyle{definition}
\newtheorem{definition}[subsection]{Definition}
\newtheorem*{remark*}{Remark}
\newtheorem*{warn*}{A word of warning}

\theoremstyle{plain}

\newcommand{\floor}[1]{\left\lfloor #1 \right\rfloor}
\newcommand{\ceiling}[1]{\left\lceil #1 \right\rceil}
\usepackage{scalerel}

\newcommand{\beqq}{\begin{align*}}
	\newcommand{\eeqq}{\end{align*}}

\author[A. Fragkos]{Anastasios Fragkos}
\author[B. Krause]{Ben Krause}
\author[M. Lacey]{Michael Lacey}
\title{Endpoint Estimates for Stein's Purely Quadratic Carleson Operator}
\begin{document}
	\maketitle
    \begin{abstract}
        	We study the near $L^1$ behavior of the maximally quadratically modulated Hilbert transform \[ \C_2f(x)  \coloneqq \sup_{\lambda} \left|\operatorname{p.v.} \int_{\R}f(x-y)e^{2 \pi i \lambda y^2} \frac{\d y }{y} \right| \] and its lacunary counterpart obtained by restricting $\lambda$ to $2^{\Z}.$
			We prove that if $\Phi$ is a Young function satisfying $	\Phi(t)=o\bigl(t\log_2t\bigr)$ then $\C_{2,\mathsf{lac}}$, and therefore $\C_2$ as well, does not satisfy a corresponding $\Phi$-modular estimate. In particular, neither the lacunary nor the full quadratic operator is of weak type $(1,1)$. In the positive direction, we establish an $L\log_1 L$ modular estimate for $\C_2$ and a $L (\log_2L)^2 \log_4L$  estimate for $\C_{2,\mathsf{lac}}.$
    \end{abstract}

    \section{Introduction}The maximally monomially modulated Hilbert transform
	\[
	\mathcal C_m f(x)
	=
	\sup_{\lambda\in\mathbb R}
	\left|
	\operatorname{p.v.}\int_{\mathbb R}
	f(x-y)e^{2 \pi i\lambda y^m}\frac{dy}{y}
	\right|
	\]
	belongs to the family of maximally modulated singular integrals arising from Carleson's theorem. For $m=1$, it is essentially the classical Carleson operator \cite{Carleson1966}. For $m\geq 2$, it is a special case of an operator considered by Stein and Wainger \cite{SteinWainger2001}, who proved $L^p$ bounds for maximal singular integrals modulated by polynomial phases with no linear term. The more difficult problem of allowing linear and higher-order terms simultaneously led to the polynomial Carleson problem: the quadratic case and the full one-dimensional polynomial Carleson conjecture were proved by Lie \cites{Lie09,Lie2020}, and subsequently the higher-dimensional case  (with more general Calder\'{o}n-Zygmund kernels) was established by Zorin-Kranich \cite{ZORINKRANICH2021107832}. Highly restricted discrete analogues of $\C_2$ were treated by two of us \cite{KL17}, with the full $\mathcal{C}_m$ operators treated by one of us and Roos \cites{KrauseRoos2022,KR23}, and finally the discrete analogue of the full Stein-Wainger operator in all dimensions was addressed by the second named author \cite{Krause24}. 
	
	While the $L^p$ behavior of the aforementioned operator is well established in the reflexive range, the precise endpoint behavior of $\{ \C_m \}$ are not known to date. The question of $\C_{1}$ is the most famous, and relates to the issue of pointwise convergence of Fourier series near $L^1$
   with the state of the art result due to the first author, joint with Di Plinio \cite{DPF25}; but $\mathcal{C}_1$ is a completely different type of operator than $\{ \mathcal{C}_m, \ m \geq 2 \}$, due to its modulation invariant, as opposed to oscillatory, nature, and so the analysis below will not extend to this case. With this in mind, we state the main results of this article; while it seems likely that our arguments below will extend to the case of $m \geq 2$, our focus is $\C_2.$  Our first result establishes the failure of any weak (modular) Orlicz type bound on $\C_2$ whenever the pertaining Young function grows slower than $ t \log_2(t)$ as $t \to \infty;$ here and throughout we define
    \begin{align}
        \log_1(t) := \log(10 + t), \qquad \log_j(t) := \log_1 \big( \log_{j-1}(t) \big).
    \end{align}
    In particular, this implies that $\C_2$ cannot be of weak-type (1,1). In fact, we show something stronger: we prove that the aforementioned phenomenon persists even if one restricts the set of modulation parameters to $2^{\Z}.$ 
	
	\begin{theorem} \label{t:loglogfail} Suppose $\Phi(t)=o(t \log_2 t)$. Then $\C_{2,\mathsf{lac}}$ does not satisfy a $\Phi$-modular estimate. In particular, for any $0 < \kappa < 1$, there exists an $f_\kappa$ so that
    \begin{align}
        \int_{\mathbb{R}} \Phi(|f_{\kappa}(x)|) \d x < \kappa \cdot |\{ x :  \C_{2,\mathsf{lac}}\left(f_\kappa(x)\right) > 1\}|.
    \end{align}
	\end{theorem}

	In the opposite direction, we are able to establish certain modular estimates for the following near-$L^1$ spaces.
	\begin{theorem} \label{t:LlogL} 
		\[  \left| \left\{\C_2 f > \alpha\right\} \right| \lesssim \int_{\R} \frac{|f(x)|}{\alpha} \log_{1}\left ( \frac{|f(x)|}{\alpha}\right ) \d x \] and \[ \left| \left\{ \C_{2,\mathsf{lac}}f >\alpha   \right\}\right| \lesssim \int_{\R}  \frac{|f(x)|}{\alpha} \log_2\left(\frac{|f(x)|}{\alpha}\right)^2 \log_4 \left(\frac{|f(x)|}{\alpha}\right) \d x  \]
	\end{theorem}
	The precise endpoint -- in particular, whether the sharp endpoint space is $L \log_2 L$, in view of Theorem \ref{t:loglogfail} -- are difficult and  attractive open questions. We mention that the space $L \log_2 L$ has already appeared in endpoint questions in Harmonic Analysis and in particular in the study of singular (maximal) Radon transforms \cite{STW04}.
	
    \subsection{Strategy and techniques} We begin this subsection with laying out the strategy for the proof of Theorem \ref{t:loglogfail}.
	
	Roughly speaking, $\C_2^{\lambda}$ acts on a wave packet adapted to some interval $Q$, call it $\varphi_Q$, in the following manner. When the modulation parameter is very large compared to the frequency scale of the wave packet, due to the ample amount of cancellation, we have the heuristic approximation  \[ \C_2^{\lambda} \varphi_Q \simeq 0,\]
    where by $\simeq$ we mean that the approximation holds up to harmless error terms.    On the other hand when $\lambda$ is very small compared to the frequency scale we are in the stationary regime and we obtain the approximation \[ \C_2^{\lambda}\varphi_Q(x) \simeq 	\frac{e(\lambda (x-c_Q)^2)}{x-c_Q}. \] Therefore, given a collection of well separated intervals $\mathcal{J}$ centered at the integers, we have that \[ \C_2^{\lambda}\left ( \frac{1}{\# \mathcal{J}} \sum_{Q \in \mathcal{J}} \varphi_Q  \right  ) \simeq \frac{e(\lambda x^2)}{\# \mathcal{J}}\sum_{\left|\ell_Q\lambda\right| \ll 1} \frac{e( \lambda c_Q^2-2c_Q \lambda x)}{x-c_Q}, \] which produces a logarithmic blow up in  the cardinality of $\mathcal{J}$ provided that the sum is long enough and that the phases \[ \left\{e(\lambda c_Q^2-2c_Q \lambda x)\right\}_{|\ell_Q \lambda | \ll 1}, \qquad \ell_Q := |Q|\] are constant up to a small error. To ensure that this logarithmic divergence happens on a sufficiently large set, we appeal to the theory of \emph{Bohr sets}, i.e.\ sets of the form
    \[ \{ x : |1 - e(\lambda x)| =o(1) \};\] in particular, we will use appropriately translated unions of almost disjoint Bohr sets corresponding to large enough frequencies.
	
	Regarding Theorem \ref{t:LlogL}, the method of proof employs a two-parameter variant of the High-Low method, with the decomposition taking place both at the level of operator, and on the function as well. Standard arguments reduce the question to the study of the oscillatory part of $\C_2$ and its lacunary variant respectively.  When the modulation parameter is sufficiently small relative to scale, i.e.\ 
    \[ |\lambda  t^2| \lesssim 1,\] the kernel 
    \[ \frac{e(\lambda t^2)}{t} \cdot \varphi(|\lambda|^{1/2} |t|), \qquad \varphi \in \mathcal{C}_c^{\infty}(\mathbb{R}) \] is a Calder\'{o}n-Zygmund kernel with a sharply quantified Calder\'{o}n-Zygmund norm in terms of Schwartz semi-norms of $\varphi$. Finally, when the modulation parameters live in $2^{\Z}$ we are able to refine our analysis substantially as, at the cancellative Calder\'{o}n-Zygmund atom scale, the number of modulation parameters that can contribute to the lower oscillatory heights of $\C_2$ is controlled, and therefore Corollary \ref{c:finitemodulationsweak11} can be used  to control the corresponding level sets via the $L^1$ norms of these atoms; this concedes the square of the double-logarithmic factor. 
	\section{Notation}
	We collect the notation used repeatedly below, proceeding from general conventions to the operator-specific definitions.
	\subsection{General conventions}
If $A,B\geq0$, we write $A\lesssim B$, or equivalently $B\gtrsim A$, if there exists an absolute constant $C>0$ such that $A\leq CB$. We write $A\simeq B$ if both $A\lesssim B$ and $B\lesssim A$. A subscript records permitted parameter dependence; for example, the implicit constant in $A\lesssim_k B$ may depend on $k$. Throughout, we denote $e(s)\coloneqq e^{2\pi i s}$.

We take $\N\coloneqq\{1,2,\ldots\}$. For $r\in\R$ and a positive
integer $H$, set $\N_{\geq r}\coloneqq\N\cap[r,\infty)$ and
$[H]\coloneqq\{1,\ldots,H\}$. In addition to that, for any sequence $\mathbf{a}=\left\{a_n\right\}_{n \in \N}$  we will write \[ \sum_{n=a}^b a_n \coloneqq  \sum_{n \in [a,b] \cap \Z }a_n \]

We write $L_0^\infty(\R)$ for the space of bounded, compactly supported measurable functions on $\R$.

We denote the torus $\R/\Z$ by $\TT$ and identify it with
$[-1/2,1/2)$. For $\xi\in\R$, define the distance from integers norm
$\|\xi\|_{\TT}\coloneqq\min_{n\in\Z}|\xi-n|$.
Thus $\|\xi\|_{\TT}=|\xi|$ when $\xi\in[-1/2,1/2)$.

\subsection{Intervals, averages, and maximal functions}
The center and length of a bounded interval $I\subset\R$ of positive length are denoted by $c_I$ and $\ell_I=|I|$, respectively. For $a>0$, the dilation $aI$ is the interval with center $c_I$ and length $a\ell_I$. We write $\cic{1}_E$ for the indicator of a measurable set $E$.

If $I\subset\R$ is a bounded interval of positive length and $0<p<\infty$, define

$$
\langle f\rangle_I
\coloneqq
\frac{1}{|I|}\int_I f(x)\,\d x,
\qquad
\langle f\rangle_{p,I}
\coloneqq
\left(
\frac{1}{|I|}\int_I|f(x)|^p\,\d x
\right)^{1/p}.
$$

For $0<p<\infty$ and $f\in L_{\mathrm{loc}}^p(\R)$, define
\[
\mathrm{M}_p f(x)
\coloneqq
\sup_{I\ni x}
\langle f\rangle_{p,I}.
\]
We write $\mathrm{M}\coloneqq\mathrm{M}_1$ for the uncentered Hardy--Littlewood maximal operator.

\subsection{Quadratic Carleson operators}
For $\lambda\in\R$, define

$$
\C_2^\lambda f(x)
\coloneqq
\operatorname{p.v.}\int_{\R}
f(x-t)\frac{e(\lambda t^2)}{t}\,\d t.
$$

The full quadratic Carleson operator and its lacunary counterpart are

$$
\C_2f(x)
\coloneqq
\sup_{\lambda\in\R}
\bigl|\C_2^\lambda f(x)\bigr|,
\qquad
\C_{2,\mathsf{lac}}f(x)
\coloneqq
\sup_{\lambda\in2^\Z}
\bigl|\C_2^\lambda f(x)\bigr|.
$$

In expressions carrying an subscript $\mathsf{i}$, the value $\mathsf{i}=\; $ denotes the full parameter set $\R$, while $\mathsf{i}=\mathsf{lac}$ denotes the lacunary parameter set $2^\Z$. The empty subscript is suppressed.

\subsection{Dyadic kernel decomposition}
Fix an even function $\phi\in\mathcal C_c^\infty(\R)$ satisfying the condition $\cic{1}_{[-1/4,1/4]}
\leq
\phi
\leq
\cic{1}_{[-1/2,1/2]}.$

For $j\in\Z$, define the odd function $
\psi_j(t)
\coloneqq
\frac{\phi(2^{-j}t)-\phi(2^{-j+1}t)}{t},
\; t\neq0.$
Then

$$
\sum_{j\in\Z}\psi_j(t)=\frac1t
\quad t\neq0, \qquad\supp(\psi_j)
\subseteq
[-2^{j-1},-2^{j-3}]
\cup
[2^{j-3},2^{j-1}].
$$

For $\lambda\in\R\setminus \left \{0\right \}$ and $r\in\Z_{\geq0}$, let $j(\lambda,r)$ be the unique integer such that $2^r
\leq
2^{j(\lambda,r)}|\lambda|^{1/2}
<
2^{r+1},$
we define the $r$-oscillatory scale piece of $\C_{2}^{\lambda}$

$$
\C_{2,r}^\lambda f(x)
\coloneqq
\int_{\R}
f(x-t)\psi_{j(\lambda,r)}(t)e(\lambda t^2)\,\d t.
$$

The oscillatory parts of the full and lacunary operators are, respectively,

$$
\C_2^{\mathsf{osc}}f(x)
\coloneqq
\sup_{\lambda\in\R\setminus\{0\}}
\left|
\sum_{r\in\Z_{\geq0}}
\C_{2,r}^\lambda f(x)
\right|, \;\; \C_{2,\mathsf{lac}}^{\mathsf{osc}}f(x)
\coloneqq
\sup_{\lambda\in2^\Z}
\left|
\sum_{r\in\Z_{\geq0}}
\C_{2,r}^\lambda f(x)
\right|.
$$

Whenever a supremum involving the pieces $\C_{2,r}^\lambda$ is written without an explicit parameter set, it is understood to range over $\R$ in the full case and over $2^\Z$ in the lacunary case. We suppress the empty subscript when referring to the full operator.

\subsection{Young functions and Orlicz spaces}

In this paper, a \emph{Young function} is a continuous, convex, and
strictly increasing function
\[
\Phi\colon[0,\infty)\longrightarrow[0,\infty)
\]
such that $\Phi(0)=0$ and either
\[
\Phi(t)=t
\qquad\text{for every }t\geq0,
\]
or
\[
\lim_{t\to\infty}\frac{\Phi(t)}{t}=\infty.
\]
The first alternative is included to cover the endpoint space $L^1(\R)$.
Throughout the paper, $\Phi$ denotes a Young function in this sense.

Every such function $\Phi$ is a bijection from $[0,\infty)$ onto itself;
we denote its inverse by $\Phi^{-1}$. Asymptotic comparisons between Young
functions are understood as $t\to\infty$. In particular,
\[
\Phi(t)=o\bigl(\Psi(t)\bigr)
\quad\Longleftrightarrow\quad
\lim_{t\to\infty}\frac{\Phi(t)}{\Psi(t)}=0.
\]

\subsection{Modular estimates}

Let $T$ be an operator mapping $L_0^\infty(\R)$ to measurable functions on $\R.$ We say that $T$ satisfies a
\emph{$\Phi$-modular estimate} if there exists a constant $C>0$ such
that
\[
\left|
\left\{
x\in\R:|Tf(x)|>\alpha
\right\}
\right|
\leq
C\int_{\R}
\Phi\left(\frac{|f(x)|}{\alpha}\right)\,\d x
\]
for every $f\in L_0^\infty(\R)$ and every $\alpha>0$. 

When $\Phi(t)=t$, the preceding inequality is precisely the weak-type $(1,1)$ estimate. 
Throughout, the term \emph{modular estimate} refers
to this global, unweighted distributional inequality.
\subsection{Sparse bounds}
Fix $\eta\in(0,1)$. A collection $\mathcal S$ of intervals in $\R$ is called \emph{$\eta$-sparse} if, for every $Q\in\mathcal S$, there exists a measurable set $E_Q\subseteq Q$ such that $\left\{E_Q\right\}_{Q \in \mathcal{S}}$ is a pairwise disjoint collection and $|E_Q|\geq\eta|Q|.$

Let $p_1,p_2\in(0,\infty)$. For an $\eta$-sparse collection
$\mathcal S$, define

$$
\mathsf A_{\mathcal S,(p_1,p_2)}(f_1,f_2)
\coloneqq
\sum_{Q\in\mathcal S}
|Q|
\langle f_1\rangle_{p_1,Q}
\langle f_2\rangle_{p_2,Q}.
$$

A scalar-valued form $\Lambda$ on
$L_0^\infty(\R)\times L_0^\infty(\R)$ is called
\emph{$(p_1,p_2)$-sparse bounded} if there exists a constant $C>0$
such that, for every $f_1,f_2\in L_0^\infty(\R)$, there is an
$\eta$-sparse collection $\mathcal S=\mathcal S(f_1,f_2)$ for which

\begin{equation} \label{eq:sparseconstant}
|\Lambda(f_1,f_2)|
\leq
C\,
\mathsf A_{\mathcal S,(p_1,p_2)}(f_1,f_2).
\end{equation}

The corresponding \emph{$(p_1,p_2)$-sparse norm}, denoted by $\left \|\Lambda\right \|_{(p_1,p_2)}$, is the infimum of all admissible constants, $C$, that satisfy the inequality \eqref{eq:sparseconstant}.

If $
T\colon L_0^\infty(\R)\longrightarrow L_{\mathrm{loc}}^1(\R)
$ is a sublinear operator, we write

$$
\|T\|_{(p_1,p_2)}
\coloneqq
\|\Lambda_T\|_{(p_1,p_2)},
\qquad
\Lambda_T(f_1,f_2)
\coloneqq
\langle T(f_1),f_2\rangle
=
\int_{\R}
T(f_1)(x)\overline{f_2(x)}\,\d x.
$$

\section{Theorem \ref{t:loglogfail}: Modular estimates fail}

We start by fixing an even positive $ \varphi \in \mathcal{C}^{\infty}(\R) $ with $\supp(\varphi) \subseteq \left [-\frac{1}{4},\frac{1}{4}\right ]$  and $\int_{\R} \varphi=1.$ Following standard notation, see for example \cite{DPWW23}, we rescale $\varphi$ to be adapted to the interval $\left[s-\frac{t}{2},s+\frac{t}{2}\right]$ where $(s,t) \in \R \times \R_{+}$ via the formula  \[ \varphi_{\left[s-\frac{t}{2},s+\frac{t}{2}\right]}(x)  \coloneqq  \varphi_{(s,t)}(x)= t^{-1}\varphi\left (\frac{x-s}{t}\right ). \]

The next lemma allows us to quantitatively understand the action of $\C_2^{\lambda}$ on the wave packet $\varphi_Q.$

	\begin{lemma}  Let $Q$ be an interval and $x \not \in Q.$ The following estimates hold:
	\[ \C_2^{\lambda}\varphi_Q(x) = \begin{cases}
		\frac{e(\lambda (x-c_Q)^2)}{x-c_Q}+O\left( \max \left\{ |\lambda| \ell_Q, \frac{\ell_Q}{|x-c_Q|^2} \right\}\right) \\
	O_k\left(\frac{1}{|\lambda|^k \ell_Q^k |x-c_Q|^{k+1}}\right)
		\\
	
	\end{cases} \] 

	\end{lemma}
	
	\begin{proof}

For the first estimate we leverage the smoothness of $\frac{e(\lambda t^2)}{t}.$ In particular, we have that \[ \begin{split}
\left|	\int_{\R} \varphi_Q(x-t) \left( \frac{e(\lambda t^2)}{t}-\frac{e(\lambda(x-c_Q)^2)}{x-c_Q} \right) \d t\right| & \lesssim \int_{\R} | \varphi_Q(x-t)| |x-t-c_Q|  \left (|\lambda|+ \frac{1}{|x-c_Q|^2}\right ) \d t \\ & \lesssim  \max \left\{ |\lambda| \ell_Q, \frac{\ell_Q}{|x-c_Q|^2} \right\},
\end{split} \] where the passage to the first inequality is possible due to the mean value theorem.

	For the second estimate, by symmetry, we may assume that $x$ lies to the right of the interval $Q$. We then have
\[\begin{split}
	\left|\int_{\R} \varphi_{(c_Q,\ell_Q)}(x-t) \frac{e(\lambda t^2)}{t} \d t \right| &  \lesssim \left|\int_{\R} \varphi_{(c_Q,\ell_Q)}(x-\sqrt{|t|}) \frac{e(\lambda t)}{t} \d t \right| \lesssim \frac{\left \| \partial^k \left( \varphi_{(c_Q,\ell_Q)}(x-\sqrt{|t|})  \frac{1}{t} \right)\right \|_{L^1(\R)}}{|\lambda|^k}  \\ & \lesssim \frac{1}{\ell_Q^k |x-c_Q|^{k+1} |\lambda|^k}
\end{split}  \] 
\end{proof}

Next, we define the notion of a Bohr set; we localize them for technical reasons.

\begin{definition} We define the Bohr set $\mathbf{B}(k,\rho)$ with frequency $k$ and radius $\rho$ as follows:
	\[
	\mathbf{B}(k,\rho)
	\coloneqq 
	\left\{
	x\in \left[ \frac{1}{4},\frac{1}{2}\right )  :
	\left \|k x\right \|_{\TT } \leq \rho	\right\}.
	\]
\end{definition} 
It is a straightforward implication of the definition to observe that for,  say, $k \geq 100$ and $\rho \in (0,\frac{1}{100})$ we have that $\left| \mathbf{B}(k,\rho)\right| \sim  \rho. $  The following lemma allows us to estimate the intersection of two Bohr sets and is expected in the following sense. Roughly speaking, $ \left \| kx \right \|_{\TT} $  quantifies  how far $x$ is from being an element of $\Z / k.$ Consequently, the intersection $\mathbf{B}(k,\rho) \cap \mathbf{B}(m,\rho)$ is maximized when $\gcd(k,m)=\min \left\{k,m\right\}.$ The following is an immediate corollary of \cite[Proposition 8]{PS2016}.
\begin{lemma} \label{l:bohrintersection}
For $ \rho \in (0,\frac{1}{10})$ and $k,m \in \N$, then	\[ \left| \mathbf{B} (k,\rho)  \cap \mathbf{B}(m,\rho) \right| \lesssim \rho^2 +\rho \frac{\gcd(k,m)}{\max\left\{k,m\right\}} \]
\end{lemma}

\begin{proposition}For large $M$ and $B\geq1$, define $\mathbf{A}(M,B,c)$ to be
	the union of Bohr sets over  $B$ many dyadic consequetive frequencies that are greater than $ M $ by the formula 
	\[ \mathbf{A}(M,B,c) \coloneqq \bigcup_{ k \in 2^{\N}, \; \log(k) \in [M,M+B)} \mathbf{B}\left (k,c\right ). \] Then, if $c \in (0,10^{-10})$ we have the following lower bound on $\mathbf{A}(M,B,c)$   \[\left| \mathbf{A}(M,B,c)\right|   \gtrsim \frac{(Bc)^2}{ (Bc)^2+Bc  }.\]   In particular, when  $c \gtrsim B^{-1}$ we have a uniform lower bound on $\mathbf{A}(M,B,c)$, namely $\left|\mathbf{A}(M,B,c) \right|\gtrsim 1.$

\end{proposition}
	\begin{proof} 

		By Lemma \ref{l:bohrintersection} we learn that  \[ \sum_{k,m \in 2^{\N}, \; \log(k),\log(m) \in [M,M+B)}|\mathbf{B}(k,c) \cap \mathbf{B}(m,c)| \lesssim   (Bc)^2+Bc,  \] so by Cauchy-Schwartz
        \[
	|\mathbf{A}(M,B,c)| \geq  \frac{\left( \displaystyle{\sum_{\substack{k \in 2^{\N} \\  \log(k) \in [M,M+B)}}} |\mathbf{B}(k,c)| \right)^2}{ \displaystyle{\sum_{\substack{k,m \in 2^{\N} \\  \log(k),\log(m) \in [M,M+B)}} |\mathbf{B}(k,c) \cap \mathbf{B}(m,c)| }} \gtrsim  \frac{(Bc)^2}{ (Bc)^2+Bc  }.	\] 
\end{proof}

We turn to the proof.
	\begin{proof}[Proof of Theorem \ref{t:loglogfail}]   Fix $N \in \N$ large, and define $\mathcal{J}_N$ the collection of intervals and the corresponding superposition of wave packets adapted to these intervals  \[\mathcal{J}_N \coloneqq  \left\{ j+ 2^{-Nj} \left[-\frac{1}{2},\frac{1}{2}\right]: j \in [N] \right\},\quad \chi_N=\frac{1}{N}\sum_{Q \in \mathcal{J}_N} \varphi_Q. \] 
It is rather  immediate that 
 \begin{equation}\label{eq:modularcalc} \int_{\R} \Phi\left( \frac{N | \chi_N(x)|}{\log N} \right) \d x \lesssim \sum_{j=1}^N 2^{-Nj} \Phi\left(\frac{C2^{Nj}}{\log N}\right) \lesssim \sum_{j=1}^N 2^{-Nj}o\left( \frac{2^{Nj}}{\log N} \log(Nj)  \right)=o(N ).
\end{equation} 
 Following our heuristics as outlined in the introduction we study the level set \[\left\{ x: \C_{2,\mathsf{lac}} \chi_N(x) \geq \mathbf{c} \frac{\log N}{N}\right\},\]
 where $\mathbf{c} > 0$ is a suitably small positive constant, which we will be free to decrease finitely many times throughout the argument as needed. The component of $\chi_N$ that corresponds to very large scales compared to the modulation parameter we select will be handled via the oscillatory estimate, while the rest of the scales are handled by the stationary estimate. With $0 < A < \infty$ a parameter to be optimized below, whenever 
\begin{align}\label{e:xrest} x \not \in \bigcup_{Q \in \mathcal{J}_N}3Q, \qquad \| x \|_{\mathbb{T}} \gtrsim 1 \end{align}
we decompose
\[ \C_2^{\lambda} \chi_N(x)= \frac{1}{N}\left(\sum_{\ell_Q |\lambda| \leq A}\C_2^{\lambda} \varphi_Q(x)+ \sum_{\ell_Q |\lambda| > A}\C_2^{\lambda} \varphi_Q(x)\right); \]
we may bound  \begin{equation} \label{eq:largemod}
	\left| \sum_{\ell_Q |\lambda| > A}\C_2^{\lambda} \varphi_Q(x)\right| \lesssim \sum_{ 2^{Nj}<\frac{|\lambda|}{A} } \frac{2^{2Nj}}{|\lambda|^2 |x-j|^3 } \lesssim \frac{1}{A^2 \left \| x \right \|_{\TT}^3} \lesssim \frac{1}{A^2},
\end{equation}
while for the smaller scales we have that \begin{equation} \label{eq:stationary} \begin{split}
		\sum_{\ell_Q |\lambda| \leq A} \C_2^{\lambda}\varphi_Q(x)&= \sum_{\ell_Q |\lambda| \leq A} \frac{e(\lambda (x-c_Q)^2)}{x-c_Q}+O \left(\max \left\{A,\frac{A}{|\lambda| \left \| x \right \|_{\TT}^2}\right\}\right)\\ &= e(\lambda x^2) \sum_{ j =\max \left\{1,\ceiling{\frac{1}{N} \log (|\lambda|/A)} \right\}}^{N}  \frac{e(j^2 \lambda-2\lambda jx  )}{x-j}+ O\left(\max\left\{A,\frac{A}{|\lambda|}\right\} \right)
	\end{split}
\end{equation} 
At this point, in order to absorb the error terms into the height of the level set  we make the choice  $A =c_0   \log N, $ for some $c_0 = c_0(\mathbf{c}) >0$ very small,  and  to realise an almost length $\sim N$  harmonic sum for the main term we restrict to the interval  $\left[  \frac{1}{A_1} 2^{\frac{1}{2}N^2}, A_1 2^{\frac{3}{4}N^2}\right] \cap 2^{\N} $, where $A_1 = A_1(\mathbf{c})$, so that for $x$ satisfying \eqref{e:xrest}, we have the pointwise estimate \[\begin{split}
&	 \sup_{ \lambda \in \left[  \frac{1}{A_1} 2^{\frac{1}{2}N^2},A_1 2^{\frac{3}{4}N^2}\right] \cap 2^{\N}} \left|\C_2^{\lambda}\chi_N- \frac{e(\lambda x^2)}{N}\sum_{ j =j_{N,\lambda} }^{N}  \frac{e(j^2 \lambda-2\lambda jx  )}{x-j}  \right| \\
& \qquad \qquad \lesssim \max\left\{\frac{\log N}{N}, \frac{1}{N A^2},\frac{A}{ 2^{\frac{1}{2}N^2}}\right\}, \qquad j_{N,\lambda} \coloneqq \frac{1}{N} \log \left(\frac{|\lambda|}{A}\right);
\end{split} \] this leads to the containment
\begin{align}
&\left\{x:  \sup_{ \lambda \in \left[  \frac{1}{A_1} 2^{\frac{1}{2}N^2},A_1 2^{\frac{3}{4}N^2}\right] \cap 2^{\N}}\left|\frac{1}{N}\sum_{ j = j_{N,\lambda} }^{N}  \frac{e(2\lambda jx  )}{x-j}\right| \gtrsim \mathbf{c} \frac{\log N}{N}, \; \left \| x \right \|_{\TT} \gtrsim 1 \right\} \\
& \subseteq \left\{x: \C_{2,\mathsf{lac}}\chi_N(x) \geq \mathbf{c} \frac{\log N}{N}\right\}\cup \left( \bigcup_{Q \in \mathcal{J}_N}3Q\right) . \end{align} It is straightforward to see that $\floor{j_{N,\lambda}}=k$ is equivalent to the membership $\lambda \in \left[2^{Nk}A,2^{N(k+1)} A \right).$ We define $E_{N,k} \coloneqq k+\mathbf{A}\left (Nk+\log(A),N,\frac{1}{2^{100}N} \right )$ and collect 
\[ E_N := \bigcup_{k \in [\frac{3N}{5},\frac{7N}{10}] \cap \N } E_{N,k};\]
note the lower bounds  \[ \inf_{x \in E_N}\sup_{ \lambda \in \left[  \frac{1}{A_1} 2^{\frac{1}{2}N^2},A_1 2^{\frac{3}{4}N^2}\right] \cap 2^{\N}} \left|\frac{1}{N}\sum_{ j = j_{N,\lambda} }^{N}  \frac{e(2\lambda jx  )}{x-j}\right|  \gtrsim \mathbf{c} \frac{\log N}{N}, \qquad \inf_{x \in E_N}  \left \| x \right \|_{\TT} \geq \frac{1}{4}. \]  Indeed if $x \in E_N$ we find $k \in \left[0.6N,0.7N\right]$ such that $x =k+\tau$ with $\tau \in \mathbf{B}(\lambda_0,\frac{1}{2^{100}N})$ for some $\lambda_0 \in 2^{\N}$ with 
\[ \lambda_0 \in \left[2^{Nk} A,2^{N(k+1)}A\right) \subseteq \left[  \frac{1}{A_1} 2^{\frac{1}{2}N^2},A_1 2^{\frac{3}{4}N^2}\right].\]  Therefore, 
\begin{align}
    \left|\frac{1}{N} \sum_{j=j_{N,\lambda_0}}^N \frac{e(2 \lambda_0 j x)}{x-j}\right|&= \left|\frac{1}{N} \sum_{j=k+1}^N \frac{e(2 \lambda_0 j x )}{k+\tau-j}\right|+O\left (\frac{1}{N}\right )\\
    & \qquad = \left|\frac{1}{N} \sum_{j=k+1}^N \frac{1}{k+\tau-j}\right|+O \left(2^{-99}\frac{\log N}{N}\right) \gtrsim A_0 \frac{\log N}{N}. 
\end{align} where the penultimate estimate is due to the fact that $ \left \| 2 \lambda_0 j x \right \|_{\TT} \leq \frac{j}{2^{99}N}  $ and the last inequality follows for example from the integral test. We arrive at the lower bounds
\[ \begin{split}
	&\left| \left\{x:  \sup_{ \lambda \in \left[  \frac{1}{A_1} 2^{\frac{1}{2}N^2},A_1 2^{\frac{3}{4}N^2}\right]  \cap 2^{\N}}\left|\frac{1}{N}\sum_{ j = j_{N,\lambda} }^{N}  \frac{e(2\lambda jx  )}{x-j}\right|  \gtrsim \frac{\log N}{N} \right\} \right|  \\
    & \qquad \qquad \qquad \geq \left|E_N\right|  \geq  \sum_{k \in \left[0.6N,0.7N\right] } \left|\mathbf{A}\left (Nk+\log(A),N,\frac{1}{2^{100}N} \right )\right| \gtrsim N. 
\end{split}\] 
This completes the proof: it is impossible to have  
\[  \left|\left\{\C_{2,\mathsf{lac}} \chi_N > \frac{\log N}{N}  \right\} \right| \leq K  \int_{\R} \Phi\left( \frac{N}{\log N} \chi_N(x) \right) \d x \] uniformly in $N$ for any $0 < K < \infty$, as this would force $N \lesssim o(N).$
	\end{proof}
	\section{Positive results }We begin this section with a technical lemma concerning the weak type $(1,1)$ norm of a finite maximal operator formed by sublinear operators whose sparse $(1,p)$ norm is controlled by $p'$ as $ p \to 1^{+}.$
	\begin{lemma} \label{l:weak11sparse}Let $ \left\{T_j\right\}_{j=1}^N$ be sublinear operators that act on $L_{0}^{\infty}(\R)$  whose sparse $(1,p)$ norm grows at most as $p'$, namely \[ \displaystyle{\sup_{p \in (1,2)} \max_{j=1,\ldots,N}} \frac{\left \| \, 
    \left|T_j\right| \, \right \|_{(1,p)}}{p'} \lesssim 1\]  then we have that  \[\left \| \displaystyle{\max_{j=1,\ldots N} }|T_j| \right \|_{L^1(\R) \to L^{1,\infty}(\R)} \lesssim  \left( \log_1 N \right)^2. \]
	\end{lemma}
	
	\begin{proof}
	As in the proof of 	\cite[Theorem E]{ACDPO17} we use the dual formulation of the weak $L^1$ norm and show that \[ \sup_{\left \| f \right \|_{L^{1}(\R)}=1} \sup_{\substack{E\subset \R \\ 0<|E|<\infty}} \inf_{\substack{E' \subset E \\ |E'| \geq \frac{1}{2}|E| }} \sup_{|g| \leq \cic{1}_{E'}} \left|\left \langle \left|T_{j( \cdot)}f\right|,g \right \rangle\right| \lesssim  \left( \log_1 N \right)^2  \] where $j: \R \to [N]$ is a function that linearizes the relevant maximal operator.  Following \cite{ACDPO17} we decompose the level-set of the Hardy-Littlewood maximal function as a disjoint union of maximal dyadic intervals,
    \[ H \coloneqq \left\{x: \mathrm{M}f >C |E|^{-1}\right\} \qquad \mathcal{Q} \coloneqq \left\{ \textbf{maximal }  Q \in \mathcal{D}: \left \langle 1_H \right \rangle_{1,Q} \geq 2^{-5}  \right\}, \]
    excise
    \[ E' \coloneqq  E \setminus \widetilde{H}, \qquad \widetilde{H}\coloneqq \bigcup_{Q \in \mathcal{Q}}3Q \]
    and foliate
 \[ E_j' \coloneqq  E' \cap \left\{x:j(x)=j\right\}, \qquad g_j \coloneqq  \cic{1}_{E_j'}(x). \] Using the hypothesis of our lemma, we extract sparse collections $\mathcal{S}_j^{p}$ such that \begin{align} \left| \left \langle \left|T_jf\right|,g_j \right \rangle \right|  &\lesssim p'  \sum_{I \in \mathcal{S}_j^p} |I| \left \langle f \right \rangle_{1,I} \left \langle g_j \right \rangle_{p,I} \\
 & \qquad  \lesssim p'  \left \| \mathrm{M} f \right \|_{L^{r'}(H^c)}\left \| \mathrm{M}_pg_j \right \|_{L^r(\R)} \lesssim p' r^{\frac{1}{r'}} \left(\frac{r}{r-p}\right)^{\frac{1}{p}} |E|^{-\frac{1}{r}} |E_j'|^{\frac{1}{r}},  \end{align}
 so
 \begin{align} \left| \left \langle \max_{j=1,\ldots N} \left|T_jf\right|,g \right \rangle \right| &\leq \sum_{j=1}^N \left| \left \langle \left|T_jf\right|,g_j \right \rangle \right| \\
 & \qquad  \lesssim p'r^{\frac{1}{r'}} \left(\frac{r}{r-p}\right)^{\frac{1}{p}} |E|^{-\frac{1}{r}} \sum_{j=1}^N |E_j'|^{\frac{1}{r}} \lesssim p' r^{\frac{1}{r'}} \left(\frac{r}{r-p}\right)^{\frac{1}{p}}N^{\frac{1}{r'}},  \end{align} where $r>p.$ Therefore \[ \left \| \max_{j=1,\ldots N} \left|T_j\right| \right \|_{L^1(\R) \to L^{1,\infty}(\R)} \lesssim \inf_{p \in (1,2)}\inf_{r>p} p' r^{\frac{1}{r'}} \left(\frac{r}{r-p}\right)^{\frac{1}{p}}  N^{\frac{1}{r'}} \lesssim  \left(\log N\right)^2  \] where the passage to the last inequality can be seen by taking  for example $p= \frac{\log N}{\log N-2}$ and $r=\frac{\log N}{\log N-3}.$
\end{proof} 
	We record the following immediate corollary. 
	\begin{corollary}\label{c:finitemodulationsweak11}
		\[  \sup_{B \geq 0 }\left \| \max_{\theta \in \Theta} \left| \sum_{0 \leq r \leq B}  \C_{2,r}^{\theta}f \right| \right \|_{L^{1,\infty}(\R)} \lesssim \left(\log_1(\#\Theta)\right)^2 \left \| f \right \|_{L^1(\R)}. \]
	\end{corollary}
	\begin{proof} Initially we observe that  \[ \left|\sum_{0 \leq r \leq B}\C_{2,r}^{\theta}f\right| \lesssim H_{*,2}^{\theta}f+\mathrm{M}f, \quad H_{*,2}^{\theta}f \coloneqq \sup_{\eps>0} \left| \int_{|t|> \eps}f(x-t) \frac{e(\theta t^2)}{t} \d t  \right|  \] and due to dilation invariance $\displaystyle{\sup_{\lambda \in \R}} \left \| H_{*,2}^{\lambda}\right \|_{(1,p)} $ can be evaluated by taking $\lambda = 1$.
    Therefore, because of \cite[Theorem 1.1]{KL18} we have that $\displaystyle{\sup_{\lambda \in \R }}\left \|\sum_{0 \leq r \leq B} \C_{2,r}^{\lambda}\right \|_{(1,p)} \lesssim p'.$ An application of Lemma \ref{l:weak11sparse} finishes the proof.
	\end{proof}
\subsection{ $L \log_1 L $ and  $L \left(\log_2 L\right)^2 \log_4L$ estimates}

We recall Theorem \ref{t:LlogL} for convenience.
\begin{theorem} The following bounds hold, uniformly for all $0 < \alpha < \infty$: 
	\[  \left| \left\{\C_2 f > \alpha\right\} \right| \lesssim \int_{\R} \frac{|f(x)|}{\alpha} \log_{1}\left ( \frac{|f(x)|}{\alpha}\right ) \d x \] and \[ \left| \left\{ \C_{2,\mathsf{lac}}f >\alpha   \right\}\right| \lesssim \int_{\R}  \frac{|f(x)|}{\alpha} \log_2\left(\frac{|f(x)|}{\alpha}\right)^2 \log_4 \left(\frac{|f(x)|}{\alpha}\right) \d x  \]
\end{theorem}
\begin{proof} 
	It is standard that  \[ \C_{2,\mathsf{i}}f \lesssim  \mathrm{M}f+H^{*}f+ \C_{2,\mathsf{i}}^{\mathsf{osc}}f , \;  \mathsf{i} \in \left\{~,\mathsf{lac}\right\} \]

Since the first two terms are weak type $(1,1)$ operators we may focus on the last term, and by homogeneity we may normalize $\alpha=1.$ We initially perform a Calder\'{o}n-Zygmund decomposition at height $ 1$ to obtain the decomposition $f=g+b$ where $g$ and $b$ enjoy the properties \begin{align} 
\left \| g \right \|_{L^{2}(\R)} &\lesssim  \left \| f \right \|_{L^1(\R)}^{\frac{1}{2}} \\
b&= \sum_{I \in \mathcal{I}
} b_I, \qquad  b_I \coloneqq \cic{1}_I \left( f- \left \langle f \right \rangle_{I} \right), \qquad  \sum_{I \in \mathcal{I}}\ell_I \lesssim \left \| f \right \|_{L^1(\R)}  \end{align} and the $\mathcal{I}$ is a disjoint collection of intervals. 
Using the $L^2(\R)$ boundedness of $\sup_{\lambda \in \R}\left|\mathcal{C}_{2,\mathsf{i}}^{\lambda}\right|$ for $g$ we can focus our attention on $b.$ Due to the fact that $f $ is not quantitatively in $L_{\mathsf{loc}}^2(\R)$ we need to further decompose $f$ and therefore $b_I.$ To wit, if we set
\[ F_k \coloneqq  \begin{cases} \left\{x: A_{k-1}<|f(x)| \leq A_k\right\}, \; \; k \in \N \\  \left\{x: |f(x)| \leq A_0\right\}, \; k=0
	\end{cases} \]
then we decompose
\[ \begin{split}
b_I= \sum_{k \geq 0 } b_{I,k}, \qquad  b_k \coloneqq \sum_{I \in \mathcal{I}} b_{I,k}, \qquad b_{I,k} \coloneqq \cic{1}_I \left( f \cic{1}_{F_k}- \left \langle f \cic{1}_{F_k} \right \rangle_I \right).
\end{split}  \] 
	For the forthcoming High-Low decomposition, the key point is that, as discussed above, the kernel 
    \[ \text{p.v. } \frac{e(\lambda t^2)}{t} \varphi(|\lambda|^{1/2} |t|), \qquad \varphi \in \mathcal{C}_c^{\infty}(\mathbb{R}) \] is a Calder\'{o}n-Zygmund kernel. Therefore, standard Calder\'{o}n-Zygmund theory allows us to get a $L^1(\R)$ estimate on the action of this kernel on a cancellative atom. In addition to this, the atoms with large modulus should be incorporated within the highly oscillatory part of the operator so that their large $L^2(\R)$ norm is offset by the oscillatory decay of $\sup_{\lambda}\left|\C_{2,r}^{\lambda}\right|.$ Motivated by the aforementioned observation, with $\mathsf{A}, \mathsf{B}$ increasing and unbounded sequences of (large) positive numbers to be determined below,
\[     
    \mathsf{A} \coloneqq \left\{A_k\right\}_{k \in \Z_{\geq 0}}, \; \mathsf{B}  \coloneqq \left\{B_k\right\}_{k \in \Z_{\geq 0} },
   \] 
    we control $\C_{2,\mathsf{i}}^{\mathsf{osc}}b$ as follows \[ \begin{split}
		\C_{2,\mathsf{i}}^{\mathsf{osc}}b & \leq \sum_{ k \geq 0 } \sup_{\lambda} \left| \sum_{0 \leq r \leq B_k} \C_{2,r}^{\lambda} \left( \sum_{I \in \mathcal{I}}b_{I,k} \right)+\sum_{r> B_k} \C_{2,r}^{\lambda} \left( \sum_{I \in \mathcal{I}}b_{I,k} \right) \right| \\ &  \leq  \sum_{k \geq 0 } \sup_{\lambda}\left|\sum_{r> B_k} \C_{2,r}^{\lambda} \left( \sum_{I \in \mathcal{I}}b_{I,k} \right)\right|+ \sum_{k \geq 0 } \sup_{\lambda} \left|\sum_{0 \leq r \leq B_k} \C_{2,r}^{\lambda} \left( \sum_{I \in \mathcal{I}}b_{I,k} \right)\right| \\ & \coloneqq \mathsf{High}_{\mathsf{A,B}}^{\mathsf{i}}(b)+\mathsf{Low}^{\mathsf{i}}_{\mathsf{A,B}}(b), \qquad  \mathsf{i} \in \left\{\;, \mathsf{lac}\right\}.
	\end{split}  \]

First, we focus on the full modulation case, where we suppress the $i$;
we begin with the $\mathsf{High}_{\mathsf{A,B}}$ term.   Since, by \cite[Theorem 1, equation (4.1)]{SteinWainger2001}, there exists $\beta>0$ so that 
\[\left  \| \sup_{\lambda} \left| \C_{2,r}^{\lambda}\right|\right \|_{L^2(\R) \to L^2(\R)} \lesssim 2^{-\beta r},\]
we may apply $L^2$-estimates to bound
\[ \begin{split}
	\left| \left\{x: \mathsf{High}_{\mathsf{A,B}}b \gtrsim 1 \right\} \right| & \lesssim \left \| \mathsf{High}_{\mathsf{A,B}}b \right \|_{L^2(\R)}^2  \lesssim \left(\sum_{k \geq 0 } \sum_{r > B_k} 2^{-\beta r} \left \| \sum_{I \in \mathcal{I}} b_{I,k} \right \|_{L^2(\R)}\right)^2   \lesssim \left(\sum_{k \geq 0} 2^{-\beta B_k} A_k^{\frac{1}{2}} \left \| b_{k} \right \|_{L^1(\R)}^{\frac{1}{2}}\right)^2 \\ & \lesssim \left(\sum_{k \geq 0} 2^{-2 \beta B_k} \frac{A_k}{B_k} \right) \left( \sum_{k \geq 0}  B_k \left \| b_k \right \|_{L^1(\R)}\right) \lesssim \left(\sum_{k \geq 0} 2^{-2 \beta B_k} \frac{A_k}{B_k} \right) \left( \sum_{k \geq 0}  B_k \int_{F_k}|f|(x) \d x  \right).
\end{split}    \] For the $\mathsf{Low}_{\mathsf{A,B}}$ term, we observe that the kernel $K_{\lambda,B}$ that corresponds to the lower oscillatory scales $r \in [0,B]$  \[ K_{\lambda,B}(t) \coloneqq  \sum_{0 \leq r \leq B} \psi_{j(\lambda,r)}(t) e(\lambda t^2)  \] satisfies the Calder\'{o}n-Zygmund estimates  
\[ \left| t K_{\lambda,B}(t) \right| \lesssim 1, \qquad | K_{\lambda,B}'(t)| \lesssim \min \left\{ |\lambda|, \frac{2^{2B}}{|t|^2} \right\}, \] 
 so standard arguments yield the estimate \[\begin{split}
\int_{\R \setminus \bigcup_{I \in \mathcal{I}}5I} \sup_{\lambda}\left| \int_{\R}b_{I,k}(x-t) K_{\lambda,B}(t) \d t \right| \d x & \lesssim \left \| b_{I,k} \right \|_{L^1(\R)}  \int_{\R \setminus \bigcup_{I \in \mathcal{I}} 5I }  \min \left\{ \frac{1}{|x-c_I|}, \frac{2^{2B}\ell_I}{|x-c_I|^2} \right\}    \d x \\ &  \lesssim B \left \| b_{I,k} \right \|_{L^1(\R)}.
\end{split} \]  

Therefore, \[ \begin{split}
	 \left| \left\{x:  \mathsf{Low}_{\mathsf{A,B}}(b) \gtrsim 1 \right\} \right| & \lesssim \left \| f\right \|_{L^1(\R)}+  \int_{\R\setminus \bigcup_{I \in \mathcal{I}} 5I } \left| \mathsf{Low}_{\mathsf{A,B}}(b)(x) \right| \d x \lesssim  \sum_{k \geq 0 }  \sum_{I \in \mathcal{I}} B_k \left \| b_{I,k} \right \|_{L^1(\R)} \\ & \lesssim \left \| f \right \|_{L^1(\R)}+ \sum_{k \geq 0} B_k \left \| b_k \right \|_{L^1(\R)} \lesssim \left \| f \right \|_{L^1(\R)} + \sum_{k \geq 0 } B_k \int_{F_k} |f|(x) \d x;
\end{split}  \]
optimizing
\[ 2^k \sim B_k=C  \log_1(A_k), \qquad C = O(1)\]
yields the desired bound.

At this point, we turn our focus to $\mathsf{i}=\mathsf{lac};$ we will re-optimize our constants $A_k,B_k$ below.

We keep the same High part and we will concentrate on the low part, where we can perform a more refined analysis when the modulation parameters live inside $2^{\Z}.$  For the remainder of the proof we write
\[ T_{k,m}h\coloneqq \sum_{0\leq r\leq B_k}\C_{2,r}^{2^m}h, \qquad b_{j,k}\coloneqq\sum_{I\in\mathcal I_j}b_{I,k}, \qquad \mathcal I_j\coloneqq\{I\in\mathcal I:\ell_I\sim2^j\}. \]
Because
\[ \supp \left(K_{2^m,B_k}\right) \subseteq \left\{ 2^{-\frac{m}{2}}  \lesssim   |t| \lesssim 2^{B_k-\frac{m}{2}}  \right\},\]
there exists an absolute constant $C_0\geq1$ such that if $T_{k,m}(b_{I,k})(x)\neq0$ for some $x\notin5I$, then $\ell_I$ cannot be too large relative to the spatial scale of $K_{2^m,B_k}$, namely
\[ 2^m\leq C_0 2^{2B_k}\ell_I^{-2}. \]

We consequently define the very small and middle ranges
\[ \begin{split}
\mathsf L_{k,m}&\coloneqq\left\{j\in\Z:2^m\leq C_0^{-1}2^{-2B_k-2j}\right\},\\
\mathsf S_{k,m}&\coloneqq\left\{j\in\Z:C_0^{-1}2^{-2B_k-2j}<2^m\leq C_0 2^{2B_k-2j}\right\}.
\end{split} \]
Then, on $\left(\bigcup_{I\in\mathcal I}5I\right)^c,$
\[ \sup_{m\in\Z}|T_{k,m}b_k|\leq \sup_{m\in\Z}\left|T_{k,m}\left(\sum_{j\in\mathsf L_{k,m}}b_{j,k}\right)\right|+\sup_{m\in\Z}\left|T_{k,m}\left(\sum_{j\in\mathsf S_{k,m}}b_{j,k}\right)\right|. \]

We first dispose of the very small range.  If $j\in\mathsf L_{k,m}$ and $I\in\mathcal I_j$, the cancellation of $b_{I,k}$ and the bound $|K_{2^m,B_k}'|\lesssim2^m$ give
\[ |T_{k,m}b_{I,k}(x)|\lesssim 2^m\ell_I\|b_{I,k}\|_{L^1(\R)}, \qquad x\notin5I. \]
Moreover, the set of $x$ for which the left hand side is nonzero has measure at most a constant multiple of $2^{B_k-m/2},$ and therefore
\[ \|T_{k,m}b_{I,k}\|_{L^1((5I)^c)}\lesssim 2^{B_k+m/2}\ell_I\|b_{I,k}\|_{L^1(\R)}. \]
Summing the geometric series in $m,$ we obtain
\[ \begin{split}
&\left\|\sup_{m\in\Z}\left|T_{k,m}\left(\sum_{j\in\mathsf L_{k,m}}b_{j,k}\right)\right|\right\|_{L^1(\R\setminus\bigcup_{I\in\mathcal I}5I)}\\
&\qquad\lesssim\sum_{m\in\Z}\sum_{j\in\mathsf L_{k,m}}2^{B_k+m/2+j}\sum_{I\in\mathcal I_j}\|b_{I,k}\|_{L^1(\R)}\lesssim\sum_{I\in\mathcal I}\|b_{I,k}\|_{L^1(\R)}\lesssim\int_{F_k}|f(x)|\d x.
\end{split} \]

We now turn to the middle range.  We partition the indices of the modulation by setting
\[ Q_{k,\tau}\coloneqq[\tau B_k,(\tau+1)B_k), \qquad \Z=\bigcup_{\tau\in\Z}(Q_{k,\tau}\cap\Z), \]
and we freeze the atom scales over each block via
\[ R_{k,\tau}\coloneqq\bigcup_{m\in Q_{k,\tau}\cap\Z}\mathsf S_{k,m}. \]
The set $\mathsf S_{k,m}$ is an interval of integers of length $2B_k+O(1)$ centered at $-m/2,$ hence $R_{k,\tau}$ is an interval of integers of length $ \leq 5B_k$; after sparsifying to $\tau$ in a single residue class $\mod 10^{10}$, we can assume that the centers of $R_{k,\tau}$ and $R_{k,\tau'}$ are separated by $100B_k$, whenever $\tau \neq \tau'$, and thus every $j\in\Z$ belongs to at most $1$ of the sets $R_{k,\tau}.$

If $m\in Q_{k,\tau}\cap\Z,$ define
\[ \mathsf{Err}_{k,\tau,m}(b)\coloneqq-\sum_{j\in R_{k,\tau}\setminus\mathsf S_{k,m}}b_{j,k}, \]
so that
\[ \sum_{j\in\mathsf S_{k,m}}b_{j,k}=\sum_{j\in R_{k,\tau}}b_{j,k}+\mathsf{Err}_{k,\tau,m}(b). \]
If $j\in R_{k,\tau}\setminus\mathsf S_{k,m},$ then either $j\in\mathsf L_{k,m}$ or $2^m>C_0 2^{2B_k-2j}.$ In the latter case $T_{k,m}b_{j,k}$ vanishes on $\left(\bigcup_{I\in\mathcal I}5I\right)^c$ by the support observation above, while in the former case the estimate already used for the very small range and the triangle inequality give
\[ \begin{split}
&\left\|\sup_{\tau\in\Z}\sup_{m\in Q_{k,\tau}\cap\Z}|T_{k,m}(\mathsf{Err}_{k,\tau,m}(b))|\right\|_{L^1(\R\setminus\bigcup_{I\in\mathcal I}5I)}\\
&\qquad\lesssim\sum_{m\in\Z}\sum_{j\in\mathsf L_{k,m}}2^{B_k+m/2+j}\sum_{I\in\mathcal I_j}\|b_{I,k}\|_{L^1(\R)}\lesssim\int_{F_k}|f(x)|\d x.
\end{split} \]
Consequently, on $\left(\bigcup_{I\in\mathcal I}5I\right)^c,$
\[ \begin{split}
\sup_{m\in\Z}\left|T_{k,m}\left(\sum_{j\in\mathsf S_{k,m}}b_{j,k}\right)\right|\leq&\sup_{\tau\in\Z}\sup_{m\in Q_{k,\tau}\cap\Z}\left|T_{k,m}\left(\sum_{j\in R_{k,\tau}}b_{j,k}\right)\right|\\
&+\sup_{\tau\in\Z}\sup_{m\in Q_{k,\tau}\cap\Z}|T_{k,m}(\mathsf{Err}_{k,\tau,m}(b))|.
\end{split} \]

The main term is treated via Kalton's log-convexity: \[ \begin{split}
&\left|\left\{x\in\left(\bigcup_{I\in\mathcal I}5I\right)^c:\sum_{k\geq0}\sup_{\tau\in\Z}\sup_{m\in Q_{k,\tau}\cap\Z}\left|T_{k,m}\left(\sum_{j\in R_{k,\tau}}b_{j,k}\right)\right|>1\right\}\right|\\
&\qquad\lesssim\sum_{k\geq0}\log_1(k+2)\sum_{\tau\in\Z}\left\|\sup_{m\in Q_{k,\tau}\cap\Z}\left| \cic{1}_{ \R \setminus \bigcup_{I \in \mathcal{I}}5I } T_{k,m}\left(\sum_{j\in R_{k,\tau}}b_{j,k}\right)\right|\right\|_{L^{1,\infty}(\R)}\\
&\qquad\lesssim\sum_{k\geq0}\log_1(k+2)\log_1(B_k)^2\sum_{\tau\in\Z}\left\|\sum_{j\in R_{k,\tau}}b_{j,k}\right\|_{L^1(\R)}\lesssim\sum_{k\geq0}\log_1(k+2)\log_1(B_k)^2\int_{F_k}|f(x)|\d x.
\end{split} \]
Here we first apply \cite[Theorems 3.4 and 3.6]{Kalton1981} to finite partial sums and then pass to the limit, the passage to the penultimate estimate follows from Corollary \ref{c:finitemodulationsweak11} because $\#(Q_{k,\tau}\cap\Z)\lesssim B_k,$ and the disjointness the sets $R_{k,\tau}$ (after sparsification).  Furthermore, for $\mathsf{High}^{\mathsf{lac}}_{\mathsf{A,B}}(b)$ we have that \[ \left| \left\{ x: \mathsf{High}^{\mathsf{lac}}_{\mathsf{A,B}}(b) \gtrsim 1 \right\} \right| \lesssim \left(\sum_{k \geq 0} 2^{-2\beta B_k} \frac{A_k}{\log_1(B_k)^2} \right) \left( \sum_{k \geq 0} \log_1(B_k)^2 \int_{F_k}|f(x)| \d x \right)  \] To conclude the proof, with similar considerations as in the case of the full maximal operator, we take $A_k=2^{2^{2^k}}$ and therefore $B_k= C2^{2^{k}} $  therefore \[ \begin{split}
&	\left|\left\{ x: \mathsf{High}^{\mathsf{lac}}_{\mathsf{A,B}}(b) \gtrsim 1\right\} \right| \lesssim \sum_{k \geq 0} 2^{2k} \int_{F_k}|f(x)| \d x \lesssim \int_{\R} |f(x)| \log_2(|f(x)|)^2 \d x \\ & \left| \left\{x: \mathsf{Low}_{\mathsf{A,B}}^{\mathsf{lac}}(b) \gtrsim 1 \right\} \right| \lesssim \int_{\R} |f(x)| \log_2(|f(x)|)^2 \log_4(|f(x)|) \d x; 
\end{split} \]
the proof is complete.

\end{proof}

		\bibliography{steinendpointbib}
	\bibliographystyle{amsplain}
\end{document}